\documentclass{amsart}
\usepackage{graphicx, color}
\usepackage{amscd}
\usepackage{amsmath,empheq}
\usepackage{amsfonts}
\usepackage{amssymb}
\usepackage{mathrsfs}
\usepackage[all]{xy}
\newtheorem{theorem}{Theorem}

\newtheorem{prop}[theorem]{Proposition}
\newtheorem{remark}{Remark}

\newtheorem{definition}[theorem]{Definition}
\newenvironment{proof-sketch}{\noindent{\bf Sketch of Proof}\hspace*{1em}}{\qed\bigskip}

\everymath{\displaystyle}
\newcommand{\RR}{\mathbb R}
\newcommand{\NN}{\mathbb N}

\renewcommand{\leq}{\leqslant}

\renewcommand{\geq}{\geqslant}
\begin{document}
\title[Periodic solutions for implicit evolution inclusions]{Periodic solutions for implicit evolution inclusions}
\author[N.S. Papageorgiou]{Nikolaos S. Papageorgiou}
\address[N.S. Papageorgiou]{National Technical University, Department of Mathematics,
				Zografou Campus, Athens 15780, Greece \& Institute of Mathematics, Physics and Mechanics, 1000 Ljubljana, Slovenia}
\email{\tt npapg@math.ntua.gr}
\author[V.D. R\u{a}dulescu]{Vicen\c{t}iu D. R\u{a}dulescu}
\address[V.D. R\u{a}dulescu]{Faculty of Applied Mathematics, AGH University of Science and Technology, 30-059 Krak\'ow, Poland \& Institute of Mathematics, Physics and Mechanics, 1000 Ljubljana, Slovenia \& Institute of Mathematics ``Simion Stoilow" of the Romanian Academy, P.O. Box 1-764, 014700 Bucharest, Romania}
\email{\tt vicentiu.radulescu@imfm.si}
\author[D.D. Repov\v{s}]{Du\v{s}an D. Repov\v{s}}
\address[D.D. Repov\v{s}]{Faculty of Education and Faculty of Mathematics and Physics, University of Ljubljana \& Institute of Mathematics, Physics and Mechanics, 1000 Ljubljana, Slovenia}
\email{\tt dusan.repovs@guest.arnes.si}
\begin{abstract}
We consider a nonlinear implicit evolution inclusion driven by a nonlinear, nonmonotone, time-varying set-valued map and defined in the framework of an evolution triple of Hilbert spaces. Using an approximation technique and a surjectivity result for parabolic operators of monotone type, we show the existence of a periodic solution.
\end{abstract}
\maketitle

\section{Introduction}

In this paper we study the following periodic implicit evolution inclusion
\begin{equation}\label{eq1}
	\left\{
		\begin{array}{ll}
			\frac{d}{dt}(Bu(t)) + A(t,u(t))\ni 0\ \mbox{for almost all}\ t\in T=[0,b] \\
			B(u(0)) = B(u(b)).
		\end{array}
	\right\}
\end{equation}

Problem (\ref{eq1}) is defined in the framework of an evolution triple $(X, H, X^*)$ of Hilbert spaces (see Section 2), where $B\in\mathcal{L}(X,X^*)$ and $A:T\times X\rightarrow 2^{X^*}$ is a map measurable in $t\in T$ and such that for almost all $t\in T$, $A(t,\cdot)$ is bounded and pseudo-monotone.

Implicit evolution equations were studied by Andrews, Kuttler \& Schillor \cite{1}, Barbu \cite{2}, Barbu \& Favini \cite{3}, Favini \& Yagi \cite{6}, Liu \cite{11}, and Showalter \cite{14}. However, in all these works, the operator $A$ was time-invariant and maximal monotone. Moreover, the aforementioned works treat the Cauchy problem. We are not aware of any work on implicit evolution equations treating the periodic problem. We mention also the works of Barbu \& Favini \cite{4} and DiBenedetto \& Showalter \cite{5}, treating the case where $B$ is nonlinear monotone. For this case the hypotheses and the techniques are different.

This paper is strongly influenced by Lions \cite{10}. In fact, our existence result (Theorem~7) is based on a multivalued version of a surjectivity result, which was proved for the first time for single-valued maps by Lions \cite[Theorem 1.2, p. 319]{10}, see Theorem 4 below. This way we can accommodate the multivalued nature of the map $A(t,x)$ in problem \eqref{eq1}. The fact that we allow $A(t,x)$ to be set-valued broadens significantly the applicability of our work. Now we can also treat the subdifferential of continuous but not $C^1$-convex functionals, a situation that the single-valued formulation cannot handle. In addition, the presence of the operator $B$ in the time derivative complicates the abstract setting. Since $B$ can be degenerate, this adds an additional level of difficulty in the analysis of problem \eqref{eq1} compared to the applications studied by Lions \cite[pp. 321-328]{10}. We overcome the difficulty, using the elliptic regularization technique, also first introduced by Lions.

\section{Mathematical background}

Suppose that $X$ and $Y$ are Banach spaces and $X$ is  continuously and densely embedded into $Y$. Then we know that $Y^*$ is continuously embedded into $X^*$ and if $X$ is reflexive, then the embedding of $Y^*$ into $X^*$ is also dense.

\begin{definition}\label{def1}
	By an ``evolutions triple", we mean a triple of spaces
	$$
	X\hookrightarrow H \hookrightarrow X^*
	$$
	such that $X$ is a separable reflexive Banach space, $H$ is a separable Hilbert space identified with its dual (pivot space), and $X$ is  continuously embedded into $H$. We say that $(X, H, X^*)$ is an evolution triple of Hilbert spaces, if all three spaces are Hilbert.
\end{definition}

Evidently, $H^*=H$ is continuously and densely embedded into $X^*$. By $||\cdot||$ (resp $|\cdot|,\ ||\cdot||_*$), we denote the norm of $X$ (resp. of $H, X^*$). We have
$$
|\cdot|\leq c_1||\cdot||\ \mbox{and}\ ||\cdot||_*\leq c_2|\cdot|\ \mbox{for some}\ c_1,c_2>0.
$$

We denote by $\langle\cdot,\cdot\rangle$ the duality brackets for the pair $(X^*, X)$ and by $(\cdot,\cdot)$ the inner product of $H$. We have
$$
\langle\cdot,\cdot\rangle|_{H\times X}=(\cdot,\cdot).
$$

Given an evolution triple $(X, H, X^*)$ and $1<p<\infty$, we can define the following Banach space:
$$
W_p(0,b) = \{u\in L^p(T,X):u'\in L^{p'}(T,X^{*})\}.
$$

In this definition, $\frac{1}{p}+\frac{1}{p'}=1$ and the derivative $u'$ of $u$ is understood in the sense of vectorial distributions. A function $u\in W_p(0,b)$ viewed as a function with values in $X^*$, is absolutely continuous and so $$W_p(0,b)\subseteq AC^{1,p'}(T,X^*)=W^{1,p'}((0,b), X^*).$$ Also, we know that $L^p(T,X^*)^*=L^{p'}(T,X).$ The space $W_p(0,b)$ is continuously and densely embedded into $C(T,H)$ and its elements satisfy the following integration by parts formula.

\begin{prop}\label{prop2}
	If $(X, H, X^*)$ is an evolution triple and $u,v\in W_p(0,b)$ $(1<p<\infty)$, then the mapping $t\mapsto (u(t),v(t))$ is absolutely continuous and
	$$
	\frac{d}{dt}(u(t),v(t)) = \langle u'(t),v(t)\rangle + \langle u(t),v'(t)\rangle\ \mbox{for almost all}\ t\in T.
	$$
\end{prop}

If $(X, H, X^*)$ is an evolution triple and $X$ is  compactly embedded into $H$, then $H^*=H$ is compactly embedded into $X^*$ (Schauder's theorem) and $W_p(0,b)$ is  compactly embedded into $L^p(T,H)$. For details, see Gasinski \& Papageorgiou \cite{7}.

We will use the following notions from set-valued analysis (see \cite{9}).
\begin{itemize}
	\item [(a)] If $V, W$ are Hausdorff topological spaces and $G:V\rightarrow 2^W\backslash\{\emptyset\}$ is a multivalued map, then we say that $G(\cdot)$ is ``upper semicontinuous" (``usc" for short), if for every $C\subseteq W$ closed, the set $G^-(C)=\{v\in V: G(v)\cap C\neq\emptyset\}$ is closed.
	\item [(b)] If $T=[0,b]$, $Y$ is a separable Banach space and $G:T\rightarrow 2^Y\backslash\{\emptyset\}$ is a multivalued map, then we say that $G(\cdot)$ is ``graph measurable" if
		$$
		\mbox{Gr}\,G = \{(t,y)\in T\times Y:y\in G(t)\}\in\mathcal{L}_T\otimes B(Y),
		$$
		with $\mathcal{L}_T$ being the Lebesgue $\sigma$-field of $T$ and $B(Y)$ the Borel $\sigma$-field on $Y$.
\end{itemize}

Given a Banach space, we will use the following notation
$$
P_{f(c)}(X)=\{C\subseteq Y:C\ \mbox{is nonempty, closed (and convex)}\}.
$$

Also, if $C\subseteq Y$, then we define
$$
|C| = \sup\left\{||c||_Y:c\in C\right\}.
$$

Let $Y$ be a reflexive Banach space and $A:Y\rightarrow2^{Y^*}$ a multivalued map. We say that $A(\cdot)$ is pseudo-monotone, if the following conditions are satisfied:
\begin{itemize}
	\item for every $y\in Y,\ A(y)$ is nonempty, closed, and convex;
	\item $A(\cdot)$ is bounded (that is, maps bounded sets to bounded sets);
	\item if $y_n\xrightarrow{w}y$ in $Y$, $y^*_n\xrightarrow{w}y^*$ in $Y^*$ with $y^*_n\in A(y_n)$ for all $n\in\NN$ and
		$$
		\limsup_{n\rightarrow\infty}\langle y^*_n,y_n-y\rangle_{Y^*Y}\leq 0,
		$$
		then $y^*\in A(y)$ and $\langle y^*_n,y_n\rangle_{Y^*Y}\rightarrow\langle y^*,y\rangle_{Y^*Y}$.
\end{itemize}

Any maximal monotone map $A:Y\rightarrow 2^{Y^*}\backslash\{\emptyset\}$ is pseudo-monotone (see Gasinski \& Papageorgiou \cite[pp. 331-332]{7}). As in the case of maximal monotone maps, pseudo-monotone operators exhibit nice surjectivity properties. In particular, a pseudo-monotone coercive (that is, $\frac{\inf\{\langle y^*,y\rangle_{Y^*Y}:y^*\in A(y)\}}{||y||_Y}\rightarrow+\infty$ as $||y||_{Y}\rightarrow+\infty$) map is surjective (see Gasinski \& Papageorgiou \cite[p. 326]{7}).

For dynamic problems (evolution equations), we have the following variant of the notion of pseudo-monotonicity.

\begin{definition}\label{def3}
	Let $Y$ be a reflexive Banach space, $L:D(L)\subseteq Y\rightarrow Y^*$ be a linear, maximal monotone operator and $A:Y\rightarrow 2^{Y^*}$ a multivalued map. We say that $A(\cdot)$ is ``L-pseudo-monotone", if the following conditions hold:
	\begin{itemize}
		\item [(i)] for every $y\in Y$, $A(y)\subseteq Y^*$ is nonempty, $w$-compact, and convex;
		\item [(ii)]  $A:Y\rightarrow 2^{Y^*}\backslash\{\emptyset\}$ is usc from every finite dimensional subspace of $Y$ into $Y^*$ furnished with the weak topology;
		\item [(iii)] if $\{y_n\}_{n\geq1}\subseteq D(L)$, $y_n\xrightarrow{w}y\in D(L)$ in $Y$, $L(y_n)\xrightarrow{w} L(y)$ in $Y^*$, $y^*_n\in A(y_n)$ for all $n\in\NN$, $y^*_n\xrightarrow{w}y^*$ in $Y^*$ and $\limsup_{n\rightarrow\infty}\langle y^*_n,y_n-y\rangle\leq0$, then $y^*\in A(y)$ and $\langle y^*_n,y_n\rangle_{Y^*Y}\rightarrow\langle y^*,y\rangle_{Y^*Y}$.
	\end{itemize}
\end{definition}

These operators have nice surjectivity properties. The following result can be found in Papageorgiou, Papalini \& Renzacci \cite{12} (the single-valued version of this property is due to Lions \cite{10}).

\begin{theorem}\label{th3}
	If $Y$ is a strictly convex reflexive Banach space, $L:D(L)\subseteq Y\rightarrow Y^*$ is a linear, maximal monotone operator, and $A:Y\rightarrow 2^{Y^*}$ is bounded, L-pseudo-monotone, and coercive, then $L+A$ is surjective.
\end{theorem}

\section{Periodic solutions}

In what follows, $T=[0,b]$ and $(X, H, X^*)$ is an evolution triple of Hilbert spaces. We assume that $X$ is compactly embedded into $H$ (hence so is $H^*=H$ into $X^*$). The hypotheses on the data of (\ref{eq1}) are the following:

$H(B)$: $B\in\mathcal{L}(X,X^*)$ and is symmetric and monotone.

$H(A)$: $A:T\times X\rightarrow P_{f_c}(X^*)$ is a multivalued map such that
\begin{itemize}
	\item [(i)] for all $x\in X$, the mapping $t\mapsto A(t,x)$ is graph measurable;
	\item [(ii)] for almost all $t\in T$, the mapping $x\mapsto A(t,x)$ is pseudo-monotone;
	\item [(iii)] for almost all $t\in T$ and all $x\in X$, we have
		$$
		|A(t,x)|\leq c_1(t) + c_2||x||^{p-1}
		$$
		with $c_1\in L^{p'}(T), 2\leq p<\infty$ and $c_2>0$;
	\item [(iv)] for almost all $t\in T$ and all $x\in X$, we have
		$$
		\inf\left\{\langle u^*,x\rangle: u^*\in A(t,x)\right\}\geq c_3||x||^p - c_4(t),
		$$
		with $c_3>0$ and $c_4\in L^1(T)$.
\end{itemize}

Let $J:X\rightarrow X^*$ be the duality (Riesz) map on the Hilbert space $X$. We know that $J(\cdot)$ is an isometric isomorphism (the Riesz-Fr\'echet theorem) which is monotone.
Hence for every $\epsilon>0$ we have $(\epsilon J+B)^{-1}\in \mathcal{L}(X^*,X)$. Then on $X^*$ we consider the following bilinear form
\begin{equation}\label{eq2}
	(u,v)_* = \langle (\epsilon J+B)^{-1}u,v\rangle\ \mbox{for all}\ u,v\in X^*.
\end{equation}

Hypotheses $H(B)$ imply that $(\cdot,\cdot)_*$ is an inner product on $X^*$. Let $|\cdot|_*$ denote the norm corresponding to this inner product. Clearly, $|\cdot|_*$ and $||\cdot||_*$ are equivalent norms on $X^*$. So, if $V^*$ denotes the space $X^*$ equipped with the norm $|\cdot|_*$, then $V^*$ is a Hilbert space. Using the Riesz-Fr\'echet theorem, we identify $V^*$ with its dual.

Let $A_\epsilon:T\times V^*\rightarrow P_{f_c}(V^*)$ be defined by
$$
A_\epsilon(t,v) = A(t, (\epsilon J+B)^{-1}v).
$$

Then we introduce the multivalued Nemitsky map $\hat{A}_\epsilon: L^p(T,V^*)\rightarrow 2^{L^{p'}(T,V^*)}$  corresponding to $A_\epsilon(\cdot,\cdot)$, defined by
$$
\hat{A}_\epsilon(v) = \{u\in L^{p'}(T,V^*):u(t)\in A_\epsilon(t,v(t))\ \mbox{for almost all}\ t\in T\}.
$$

Consider the function space $$W^{per}_p((0,b),V^*)=\{u\in L^p(T,V^*):u'\in L^{p'}(T,V^*), u(0)=u(b)\}.$$ We know that $W^{per}_p((0,b),V^*)\hookrightarrow C(T,V^*)$ and so the evaluations of $u$ at $t=0$ and $t=b$ make sense. Let $L:W^{per}_p((0,b),V^*)\subseteq L^p(T,V^*)\rightarrow L^{p'}(T,V^*)$ be defined by
$$
L(u)=u'.
$$

We know that $L(\cdot)$ is linear and maximal monotone (see Hu \& Papageorgiou \cite[p. 419]{9} and Zeidler \cite[p. 855]{15}).

\begin{prop}\label{prop4}
	If hypotheses $H(B), H(A)$ hold and $\epsilon>0$, then for every $u\in L^p(T,V^*)$, $\hat{A}_\epsilon(u)\subseteq L^{p'}(T,V^*)$ is nonempty, $w$-compact and convex, and the mapping $u\mapsto\hat{A}_\epsilon(u)$ is $L$-pseudo-monotone.
\end{prop}
\begin{proof}
	It is clear that $\hat{A}_\epsilon(u)$ is closed, convex and bounded, thus $w$-compact in $L^{p'}(T,V^*)$. We need to show that $\hat{A}_\epsilon(\cdot)$ has nonempty values. Note that hypotheses $H(A)(i),(ii)$ do not imply the graph measurability of $(t,x)\mapsto A_\epsilon(t,x)$ (see Hu \& Papageorgiouo \cite[p. 227]{9}). To show the nonemptiness of $\hat{A}_\epsilon(u)$ we proceed as follows. Let $\{s_n\}_{n\geq1}\subseteq L^p(T,V^*)$ be step functions such that
	\begin{eqnarray*}
	s_n\rightarrow u\ \mbox{in}\ L^p(T,V^*), s_n(t)\rightarrow u(t)\ \mbox{for almost all}\ t\in T, \\
	|s_n(t)|_*\leq |u(t)|_*\ \mbox{for almost all}\ t\in T,\ \mbox{and for all}\ n\in\NN.
	\end{eqnarray*}
	
	On account of hypothesis $H(A)(i)$, for every $n\in\NN$ the mapping $$t\mapsto A_\epsilon(t,s_n(t))=A(t,(\epsilon J+B)^{-1}s_n(t))$$ is graph measurable. So, we can apply the Yankov-von Neumann-Aumann selection theorem (see Hu \& Papageorgiou \cite[p. 158]{9}) and obtain that $v_n:T\rightarrow V^*$ is measurable and $v_n(t)\in A_\epsilon(t, s_n(t))$ for almost all $t\in T, n\in\NN$. Evidently, $v_n\in L^{p'}(T,V^*)$ and $\{v_n\}_{n\geq1}\subseteq L^{p'}(T,V^*)$ is bounded. So, by passing to a suitable subsequence if necessary we may assume that
	\begin{equation}\label{eq3}
		v_n\xrightarrow{w}v\ \mbox{in}\ L^{p'}(T,V^*)\ \mbox{as}\ n\rightarrow\infty.
	\end{equation}
	
	Note that the pseudo-monotonicity of $A_\epsilon(t,\cdot)$ (see hypothesis $H(A)(ii)$) implies that $\mbox{Gr}\,A_\epsilon(t,\cdot)$. is demiclosed (that is, sequentially closed in $V^*\times V^*_w$, where $V^*_w$ denotes the Hilbert space $V^*$ furnished with the weak topology). So, by (\ref{eq3}) and Proposition 3.9 of Hu \& Papageorgiou \cite[p. 694]{9}, we have
	\begin{eqnarray*}
		& v(t)\in \overline{\rm conv}\, w-\limsup_{n\rightarrow\infty} A_\epsilon(t, s_n(t))\subseteq A_\epsilon(t,u(t))\ \mbox{for almost all}\ t\in T, \\
		&\Rightarrow  v\in \hat{A}_\epsilon(u)\ \mbox{and so}\ \hat{A}_\epsilon(\cdot)\ \mbox{has nonempty values}.
	\end{eqnarray*}
	
	Next, we will prove the $L$-pseudo-monotonicity of $\hat{A}_\epsilon$. So, let $((\cdot,\cdot))_*$ denote the duality brackets for the pair $(L^{p'}(T,V^*),L^p(T,V^*))$, that is,
	\begin{equation}\label{eq4}
		((v,u))_* = \int^b_0(v(t),u(t))_*dt\ \mbox{for all}\ u\in L^{p}(T,V^*), v\in L^{p'}(T,V^*).
	\end{equation}
	
	Consider a sequence $\{u_n\}_{n\geq1}\subseteq W^{per}_p((0,b),V^*)$ such that
	\begin{equation}\label{eq5}
		\begin{array}{ll}
			``u_n\xrightarrow{w}u\ \mbox{in}\ L^p(T,V^*), u'_n\xrightarrow{w}u'\ \mbox{in}\ L^{p'}(T,V^*)\ \mbox{and}\ v_n\in\hat{A}_\epsilon(u_n)\ (\mbox{for all}\ n\in\NN), \\
			\mbox{such that}\ v_n\xrightarrow{w}v\ \mbox{in}\ L^{p'}(T,V^*)\ \mbox{and}\ \limsup_{n\rightarrow\infty}((v_n,u_n-u))_*\leq0".
		\end{array}
	\end{equation}
	
	We have
	\begin{eqnarray*}
		\begin{array}{ll}
			((v_n,u_n-u))_* &= \int^b_0(v_n(t),u_n(t)-u(t))_*dt\ \mbox{(see (\ref{eq4}))} \\
			&= \int^b_0\langle v_n(t), (\epsilon J+B)^{-1}(u_n-u)(t)\rangle dt\ \mbox{(see (\ref{eq2}))}.
		\end{array}
	\end{eqnarray*}
	
	Let $y_n(t)=(\epsilon J+B)^{-1}u_n(t),\ y(t)=(\epsilon J+B)^{-1}u(t)$. Then $y_n, y\in L^p(T,X)$ and we have
	$$
	\langle v_n(t),(\epsilon J+B)^{-1}(u_n-u)(t)\rangle = \langle v_n(t), y_n(t) - y(t)\rangle
	$$
	with $v_n(t)\in A(t, y_n(t))$ for almost all $t\in T$, all $n\in\NN$. Evidently,
	\begin{equation}\label{eq6}
		\{y_n\}_{n\geq1}\subseteq L^p(T,X)\ \mbox{is bounded (see (\ref{eq5}))}.
	\end{equation}
	
	Also, we have
	\begin{equation}\label{eq7}
		\begin{array}{ll}
			& y'_n = ((\epsilon J+B)^{-1}u_n)' \\
			\Rightarrow & \{y'_n\}_{n\geq1}\subseteq L^{p'}(T, X^*)\ \mbox{is bounded (see (\ref{eq5}))}.
		\end{array}
	\end{equation}
	
	It follows from (\ref{eq6}) and (\ref{eq7}) that
	$$
	\{y_n\}_{n\geq1}\subseteq W_p(0,b)\ \mbox{is bounded}.
	$$
	
	So, we may assume that
	\begin{equation}\label{eq8}
		y_n\xrightarrow{w} y\ \mbox{in}\ W_p(0,b)\ \mbox{as}\ n\rightarrow\infty.
	\end{equation}
	
	Evidently, we have $y=(\epsilon J+B)^{-1}u$ and so
	$$
	(\epsilon J+B)^{-1}u_n\xrightarrow{w}(\epsilon J+B)^{-1}u\ \mbox{in}\ L^p(T,X).
	$$
	
	If we denote by $((\cdot,\cdot))$ the duality brackets for the pair $(L^{p'}(T,X^*), L^p(T,X))$, that is,
	$$
	((v,u)) = \int^b_0\langle v(t),u(t)\rangle dt\ \mbox{for all}\ u\in L^p(T,X), v\in L^{p'}(T,X^*),
	$$
	then we have
	$$
	\limsup_{n\rightarrow\infty}((v_n,y_n-y)) = \limsup_{n\rightarrow\infty}((v_n,u_n-u))\leq0\ \mbox{(see (\ref{eq5}))}.
	$$
	
	Recall that $W_p(0,b)$ is continuously embedded in $C(T,H)$. So, from (\ref{eq8}) we have
	\begin{equation}\label{eq9}
		y_n(t)\xrightarrow{w} y(t)\ \mbox{in}\ H\ \mbox{for all}\ t\in T.
	\end{equation}
	
	Let $\vartheta_n(t) = \langle v_n(t), y_n(t)-y(t)\rangle$ and let $N\subseteq T$ be the Lebesgue-null set outside of which hypotheses $H(A)(ii),\,(iii)\,(iv)$ hold. Then for $t\in T\backslash N$, we have
	\begin{eqnarray}
		\vartheta_n(t)\geq c_3||y_n(t)||^p - c_4(t) - ||y(t)|| \left(c_1(t) + c_2||y_n(t)||^{p-1}\right) \label{eq10} \\
		\mbox{(see hypotheses $H(A)(iii),\,(iv)$)}. \nonumber
	\end{eqnarray}
	
	Let $E = \{t\in T:\liminf_{n\rightarrow\infty}\vartheta_n(t)<0\}$. This is a Lebesgue measurable set. Suppose that $\lambda^1(E)>0$ ($\lambda^1(\cdot)$ denotes the Lebesgue measure on $\RR$). From (\ref{eq10}), we see that $\{y_n(t)\}_{n\geq1}\subseteq X$ is bounded for all $t\in E\cap(T\backslash N)$. So, on account of (\ref{eq9}) we obtain that $y_n(t)\xrightarrow{w}y(t)$ in $X$. Fix $t\in E\cap(T\backslash N)$ and choose a suitable subsequence (depending on $t$) such that $\liminf_{n\rightarrow\infty}\vartheta_n(t)=\lim_{k\rightarrow\infty}\vartheta_{n_k}(t)$. The pseudo-monotonicity of $A(t,\cdot)$ (see hypothesis $H(A)(ii)$), implies that
	$$
	\langle v_{n_k}(t), y_{n_k}(t)-y(t)\rangle\rightarrow0,
	$$
	a contradiction since $t\in E$. Therefore $\lambda^1(E)=0$ and so we have
	\begin{equation}\label{eq11}
		0\leq\liminf_{n\rightarrow\infty} \vartheta_n(t)\ \mbox{for almost all}\ t\in T.
	\end{equation}
	
	Invoking Fatou's lemma, we have
	\begin{eqnarray}
		&& 0\leq \int^b_0\liminf_{n\rightarrow\infty} \vartheta_n(t)dt\leq\liminf_{n\rightarrow\infty}\int^b_0\vartheta_n(t)dt\leq\limsup_{n\rightarrow\infty}\int^b_0\vartheta_n(t)dt\leq0, \nonumber \\
		&\Rightarrow & \int^b_0\vartheta_n(t)dt\rightarrow\vartheta\ \mbox{as}\ n\rightarrow\infty. \label{eq12}
	\end{eqnarray}
	
	We have $|\vartheta_n|=\vartheta^+_n+\vartheta^-_n=\vartheta_n+2\vartheta^-_n$ and $\vartheta^-_n(t)\rightarrow0$ for almost all $t\in T$ (see (\ref{eq11})). Also, from (\ref{eq10}) we have
	$$
	\gamma_n(t)\leq\vartheta_n(t)\ \mbox{for almost all}\ t\in T,\ \mbox{and for all}\ n\in\NN,
	$$
	and $\{\gamma_n\}_{n\geq1}\subseteq L^1(T)$ is uniformly integrable. We have
	\begin{eqnarray*}
		&& 0\leq\vartheta^-_n(t)\leq \gamma^-_n(t)\ \mbox{for almost all}\ t\in T,\ \mbox{and for all}\ n\in\NN, \\
		&\Rightarrow & \{\vartheta^-_n\}_{n\geq1}\subseteq L^1(T)\ \mbox{is uniformly integrable}.
	\end{eqnarray*}
	
	Applying the extended dominated convergence theorem (see, for example, Gasinski \& Papageorgiou \cite[p. 901]{7}), we have
	\begin{eqnarray*}
		& \int^b_0\vartheta^-_n(t)dt\rightarrow0, \\
		\Rightarrow & \vartheta_n\rightarrow0\ \mbox{in}\ L^1(T)\ \mbox{(see (\ref{eq12}))}.
	\end{eqnarray*}
	
	So, by passing to a subsequence if necessary, we may assume that
	\begin{eqnarray*}
		& \vartheta_n(t)\rightarrow0\ \mbox{for almost all}\ t\in T, \\
		\Rightarrow & \langle v_n(t), y_n(t)-y(t)\rangle\rightarrow0\ \mbox{for almost all}\ t\in T.
	\end{eqnarray*}
	
	Since $v_n(t)\in A(t,y_n(t))$ for almost all $t\in T$ and for all $n\in\NN$, on account of the pseudo-monotonicity of $A(t,\cdot)$ (see hypothesis $H(A)(ii)$), we have
	\begin{equation*}
		v(t) = A(t,y(t)) = A_\epsilon(t,u(t))\ \mbox{for almost all}\ t\in T
	\end{equation*}
	and $v_n(t)\xrightarrow{w} v(t)$ in $X^*$, $\langle v_n(t), y_n(t)\rangle\rightarrow \langle  v(t),y(t)\rangle$ for almost all $t\in T$.
	
	By the dominated convergence theorem, we have
	\begin{eqnarray*}
		& v_n\xrightarrow{w} v\ \mbox{in}\ L^{p'}(T,X^*),\ ((v_n,y_n))\rightarrow ((v,y)),\ v\in\hat{A}(y), \\
		\Rightarrow & v_n\xrightarrow{w} v\ \mbox{in}\ L^{p'}(T,V^*),\ ((v_n,u_n))\rightarrow((v,u))_*,\  v\in\hat{A}_\epsilon(u).
	\end{eqnarray*}
	
	Finally, using Proposition 2.23 of Hu \& Papageorgiou \cite[p. 43]{9}, we easily see that $\hat{A}_\epsilon(\cdot)$ is usc from finite dimensional subspaces of $L^p(T,V^*)$ into $L^{p'}(T,V^*)_w$.
	
	Therefore we conclude that $\hat{A}_\epsilon$ is indeed $L$-pseudo-monotone.
\end{proof}

We consider the following auxiliary approximate periodic problem:
\begin{equation}\label{eq13}
	\left\{
		\begin{array}{ll}
			u'(t) + A_\epsilon(t,u(t))\ni 0\ \mbox{for almost all}\ t\in T, \\
			u(0) = u(b).
		\end{array}
	\right\}
\end{equation}

\begin{prop}\label{prop5}
	If hypotheses $H(B), H(A)$ hold and $\epsilon>0$, then problem (\ref{eq13}) has a solution $u_\epsilon\in W^{per}_p((0,b),V^*)$.
\end{prop}

\begin{proof}
	We rewrite (\ref{eq13}) as the following abstract operator inclusion
	\begin{equation}\label{eq14}
		L(u) + \hat{A}_\epsilon(u)\ni0.
	\end{equation}
	
	Let $v\in\hat{A}_\epsilon(u)$. We have
	\begin{equation*}
		((v,u))_* = ((v,(\epsilon J+B)^{-1}u)).
	\end{equation*}
	
	Let $y=(\epsilon J+B)^{-1}u$. Then $v\in \hat{A}(y)$ and so, using hypothesis $H(A)(iv)$, we have
	\begin{eqnarray}
		& ((v,y)) = \int^b_0\langle v(t),y(t)\rangle dt \geq c_3||y||^p_{L^p(T,X)} - ||c_4||_1, \nonumber \\
		\Rightarrow & ((v,u))_* \geq c_5 ||u||^p_{L^p(T,V^*)} - ||c_4||_1\ \mbox{for some}\ c_5>0 \label{eq15}
	\end{eqnarray}
	(recall that $|\cdot|_*$ and $||\cdot||_*$ are equivalent norms on $X^*$). It follows that $\hat{A}_\epsilon(\cdot)$ is coercive. Clearly it is bounded (see hypothesis $H(A)(iii)$). Also, from Proposition \ref{prop4} we know that $\hat{A}_\epsilon(\cdot)$ is L-pseudo-monotone. Since $L(\cdot)$ is maximal monotone, we can use Theorem \ref{th3} and find $u_\epsilon\in W^{per}_p((0,b),V^*)=D(L)$ such that it solves (\ref{eq14}). Evidently, this is a solution of problem (\ref{eq13}).
\end{proof}

Next, we will let $\epsilon\downarrow0$ to produce a solution of problem (\ref{eq1}).
\begin{theorem}\label{th6}
	If hypotheses $H(B), H(A)$ hold, then problem (\ref{eq1}) has a solution $y\in L^p(T,X)$ which satisfies $(By)'\in L^{p'}(T,X^*)$.
\end{theorem}
\begin{proof}
	For each $\epsilon>0$, let $u_\epsilon\in W^{per}_p((0,b), V^*)$ be a solution of the approximate problem (\ref{eq13}) (see Proposition \ref{prop5}). We have
	\begin{equation}\label{eq16}
		\left\{
			\begin{array}{ll}
				u'_\epsilon(t) + A_\epsilon(t,u_\epsilon(t))\ni0\ \mbox{for almost all}\ t\in T, \\
				u_\epsilon(0) = u_\epsilon(b).
			\end{array}
		\right\}
	\end{equation}
	
	We take the inner product in $V^*$ with $u_\epsilon(t)$. Then
	\begin{equation*}
		\frac{1}{2}\frac{d}{dt}|u'_\epsilon(t)|^2_* + (v_\epsilon(t), u_\epsilon(t))_* = 0\ \mbox{for almost all}\ t\in T,
	\end{equation*}
	with $v_\epsilon\in L^{p'}(T,V^*), v_\epsilon(t)\in A_\epsilon(t,u_\epsilon(t))$ for almost all $t\in T$. Integrating on $T$ and using (\ref{eq15}) and the periodic conditions, we obtain
	\begin{eqnarray}
		& c_5||u_\epsilon||_{L^p(T,V^*)} \leq ||c_4||_1, \nonumber \\
		\Rightarrow & \{u_\epsilon\}_{\epsilon>0} \subseteq L^p(T,V^*)\ \mbox{is bounded}. \label{eq17}
	\end{eqnarray}
	
	We set $y_\epsilon(t) = (\epsilon J+B)^{-1}u_\epsilon(t)$. Then
	\begin{eqnarray}
		& ||y_\epsilon(t)||\leq ||(\epsilon J+B)^{-1}||_\mathcal{L}||u_\epsilon(t)||^* \nonumber \\
		\Rightarrow & \{y_\epsilon\}_{\epsilon\in(0,1]}\subseteq L^p(T,X)\ \mbox{is bounded}\ (\mbox{see}\  (\ref{eq17})). \label{eq18}
	\end{eqnarray}
	
	On account of hypothesis $H(A)(iii)$, we have
	\begin{equation}\label{eq19}
		|A_\epsilon(t,u_\epsilon(t))|\leq c_1(t)+c_2||y_\epsilon(t)||^{p-1}\ \mbox{for almost all}\ t\in T.
	\end{equation}
	
	Then  it follows from (\ref{eq16}), (\ref{eq18}) and (\ref{eq19}) that
	\begin{equation*}
		\{u'_\epsilon\}_{\epsilon\in(0,1]}\subseteq L^{p'}(T,V^*)\ \mbox{is bounded}.
	\end{equation*}
	
	This together with (\ref{eq17}) implies that
	\begin{equation}\label{eq20}
		\{u_\epsilon\}_{\epsilon\in (0,1]}\subseteq W^{1,p'}((0,b), V^*)\ \mbox{is bounded (recall that $1<p'\leq2\leq p$)}.
	\end{equation}
	
	Now let $\epsilon_n=\frac{1}{n}, u_n=u_{\epsilon_n}, y_n=y_{\epsilon_n}, v_n=v_{\epsilon_n}$ for all $n\in\NN$. Note that
	\begin{equation*}
		[(n^{-1} J+B)y_n(t)]'\in L^{p'}(T,X^*).
	\end{equation*}
	
	We have
	\begin{equation}\label{eq21}
		\left\{
			\begin{array}{ll}
				((n^{-1} J+B)y_n(t))' + v_n(t)=0\ \mbox{for almost all}\ t\in T, \\
				v_n(t)\in A(t, y_n(t))\ \mbox{for almost all}\ t\in T, \\
				u_n(0) = u_n(b).
			\end{array}
		\right\}
	\end{equation}
	
	Note that
	\begin{equation}\label{eq22}
		y_n(0) = (\epsilon J+B)^{-1}u_n(0) = (\epsilon J+B)^{-1}u_n(b) = y_n(b)\ \mbox{for all}\ n\in\NN\ \mbox{(see (\ref{eq21}))}.
	\end{equation}
	
	Also, on account of (\ref{eq18}), (\ref{eq20}) and (\ref{eq21}), we may assume that
	\begin{equation}\label{eq23}
		y_n\xrightarrow{w} y\ \mbox{in}\ L^p(T,X),\ u_n\xrightarrow{w}u\ \mbox{in}\ W^{1,p'}((0,b),V^*),\  v_n\rightarrow v\ \mbox{in}\ L^{p'}(T,X^*).
	\end{equation}
	
	We know that $W^{1,p'}((0,b),V^*)\hookrightarrow C(T,V^*)$ continuously. Hence by (\ref{eq17}), up to  a subsequence, we have
	\begin{eqnarray}
		& u_n\xrightarrow{w} u\ \mbox{in}\ C(T,V^*), \nonumber \\
		\Rightarrow & y_n(t)\xrightarrow{w} y(t)\ \mbox{in}\ X\ \mbox{for all}\ t\in T, \label{eq24} \\
		\Rightarrow & B(y(0)) = B(y(b))\ \mbox{(see (\ref{eq22}))}. \label{eq25}
	\end{eqnarray}
	
	On the first equation in (\ref{eq21}) we act with $(y_n-y)(t)$ and then integrate over $T$. We obtain
	\begin{equation}\label{eq26}
		(((\left[n^{-1}J+B\right]y_n)',y_n-y)) + ((v_n, y_n-y))=0\ \mbox{for all}\ n\in\NN.
	\end{equation}
	
	We obtain
	\begin{eqnarray}
		& (((\left[n^{-1} J+B\right]y_p)',y_n-y)) \nonumber \\
		= & (((\left[n^{-1} J+B\right](y_n-y))',y_n-y)) + (((\left[n^{-1} J+B\right]y)',y_n-y)). \label{eq27}
	\end{eqnarray}
	
	Note that
	\begin{equation}\label{eq28}
		(((\left[n^{-1} J+B\right]y)',y_n-y))\rightarrow0\ \mbox{as}\ n\rightarrow\infty\ \mbox{(see (\ref{eq23}))}.
	\end{equation}
	
	Also, we have
	\begin{eqnarray}
		& (((\left[n^{-1} J+B\right](y_n-y))',y_n-y)) \nonumber \\
		&=  \int^b_0\langle n^{-1}(J(y_n-y))', y_n-y\rangle dt + \int^b_0\langle (B(y_n-y))', y_n-y\rangle dt \nonumber \\
		&=  \int^b_0\frac{1}{n}(y'_n-y',y_n-y)_X dt + \frac{1}{2}\int^b_0\frac{d}{dt}\langle B(y_n-y),y_n-y\rangle dt \nonumber \\
		& \mbox{(recall that $J(\cdot)$ is the Riesz map for $X$ and see hypothesis $H(B)$)} \nonumber \\
		&=  \frac{1}{n}\left[||(y_p-y)(b)|| - ||(y_n-y)(0)||\right] + \frac{1}{2}[\langle B(y_n-y)(b),(y_n-y)(b)\rangle - \nonumber \\
		& \langle B(y_n-y)(0),(y_n-y)(0)\rangle] \nonumber \\
		&=  0\ \mbox{for all}\ n\in\NN\ \mbox{(see (\ref{eq22}), (\ref{eq24}))}. \label{eq29}
	\end{eqnarray}
	
	So, if we return to (\ref{eq27}) and use (\ref{eq28}), (\ref{eq29}) we obtain
	\begin{equation}\label{eq30}
		\lim_{n\rightarrow\infty}(((\left[n^{-1} J+B\right]y_n)', y_n-y))=0.
	\end{equation}
	
	If we use (\ref{eq30}) in (\ref{eq26}), we get
	\begin{equation*}
		\lim_{n\rightarrow\infty}((v_n,y_n-y))=0.
	\end{equation*}
	
	Invoking Proposition \ref{prop4}, we have
	\begin{equation*}
		v\in \hat{A}(y)\ \mbox{and}\ ((v_n,y_n))\rightarrow((v,y)).
	\end{equation*}
	
	Thus,  we obtain from (\ref{eq21}) taking the limit as $n\rightarrow\infty$
	\begin{equation*}
		\left\{
			\begin{array}{ll}
				\frac{d}{dt}(By(t)) + A(t,y(t))\ni0\ \mbox{for almost all}\ t\in T, \\
				B(y(0)) = B(y(b)).
			\end{array}
		\right\}
	\end{equation*}
	
	Therefore $y\in L^p(T,X)$ is a solution of (\ref{eq1}) with $(By)'\in L^{p'}(T,X^*)$.
\end{proof}

\section{An example}

Let $T=[0,b]$ and let $\Omega\subseteq\RR^\NN$ be a bounded domain with a $C^2$-boundary $\partial\Omega$. We consider the following initial boundary value problem:
\begin{equation}\label{eq31}
	\left\{
		\begin{array}{ll}
			\frac{d}{dt}(m(z)u) - {\rm div}\,(a(t,z)Du) + \sum^N_{k=1}(\sin u) D_ku + \partial g(u)\ni0\ \mbox{in}\ T\times\Omega, \\
			u|_{T\times\partial\Omega}=0,\ m(z)u(z,0)=m(z)u(z,b)\ \mbox{for almost all}\ z\in \Omega.
		\end{array}
	\right\}
\end{equation}

We impose the following conditions on the data for problem (\ref{eq31}):

\smallskip
$H(m)$: $m\in L^{N/2}(\Omega)$ if $N>2$, $m\in L^r(\Omega)$ with $r>1$ if $N=2$ and $m\in L^1(\Omega)$ if $N=1$, $m(z)\geq0$ for almost all $z\in\Omega$, $m\not\equiv0$.

\smallskip
$H(a)$: $a\in L^\infty(T\times\Omega)$ and $a(t,z)\geq a_0>0$ for almost all $(t,z)\in T\times\Omega$.

\smallskip
$H(g)$: $g:\RR\rightarrow\RR$ is a continuous convex function and its subdifferential $\partial g(x)$ satisfies
\begin{equation*}
	|\partial g(x)|\leq \hat{c}\,(1+|x|^{p-1})\ \mbox{for all}\ x\in\RR,\ \mbox{and for some}\ \hat{c}>0,\ 2\leq p<\infty.
\end{equation*}
\begin{remark}\label{rem1}
	For any continuous convex function $g(\cdot)$, we know that $\partial g(x)\neq\emptyset$ for all $x\in\RR$ (see Gasinski \& Papageorgiou \cite[p. 527]{7}).
\end{remark}

We introduce the following multifunction
\begin{equation*}
	N_g(u) = \{v\in L^{p'}(\Omega): v(z)\in \partial g(u(z))\ \mbox{for almost all}\ z\in\Omega\}
\end{equation*}
for all $u\in H^1_0(\Omega)$. Evidently, $N_g(\cdot)$ is maximal monotone.

In this case, the evolution triple consists of the following Hilbert spaces:
\begin{equation*}
	X = H^1_0(\Omega),\ H = L^2(\Omega),\ X^* = H^{-1}(\Omega).
\end{equation*}

We know that $X\hookrightarrow H$ compactly (by the Sobolev embedding theorem).

Let $A_1:T\times X\rightarrow X^*$ be the nonlinear map defined by
\begin{eqnarray*}
	\langle A_1(t,u),h\rangle = \int_\Omega a(t,z)(Du,Dh)_{\RR^\NN}dz + \int_\Omega\sin u\left(\sum^N_{k=1}D_ku\right)hdz \\
	\mbox{for all}\ u,h\in X = H^1_0(\Omega).
\end{eqnarray*}

Then the mapping $t\mapsto A_1(t,u)$ is measurable, whereas $u\mapsto A_1(t,u)$ is pseudo-monotone (see, for example, Zeidler \cite[p. 591]{15}). We set
\begin{equation*}
	A(t,u) = A_1(t,u) + N_g(u).
\end{equation*}

Then $A(t,u)$ satisfies hypotheses $H(A)$ (see $H(a)$ and $H(g)$).

In addition, we let $B\in\mathcal{L}(X,X^*)$ be defined by
\begin{equation*}
	B u(\cdot) = m(\cdot)u(\cdot)\ \mbox{for all}\ u\in X=H^1_0(\Omega).
\end{equation*}

Clearly, $B(\cdot)$ satisfies $H(B)$.

We can rewrite problem (\ref{eq31}) as the following abstract implicit evolution inclusion:
\begin{equation*}
	\left\{
		\begin{array}{ll}
			\frac{d}{dt}(Bu(t)) + A(t, u(t))\ni0\ \mbox{for almost all}\ t\in T, \\
			B(u(0)) = B(u(b)).
		\end{array}
	\right\}
\end{equation*}

We can apply Theorem \ref{th6} and obtain the following result.
\begin{prop}\label{prop7}
	If hypotheses $H(m), H(a), H(g)$ hold, then problem (\ref{eq31}) admits a solution $u\in L^p(T,H^1_0(\Omega))$ with
	\begin{equation*}
		(Bu)'\in L^{p'}(T, H^{-1}(\Omega)).
	\end{equation*}
\end{prop}

\begin{remark}\label{rem2}
	Using the methods developed in  this paper one can also treat antiperiodic problems (see Gasinski \& Papageorgiou \cite{8}), problems with subdifferential terms (see Papageorgiou \& R\u{a}dulescu \cite{preect}), and applications to distributed parameter control systems (see Papageorgiou, R\u{a}dulescu \& Repov\v{s} \cite{13}).
\end{remark}

\medskip
{\bf Acknowledgments.} The authors thank an anonymous referee for the careful reading of this paper and for useful remarks.
This research was supported by the Slovenian Research Agency grants
P1-0292, J1-8131, J1-7025, N1-0064, and N1-0083. V.D.~R\u adulescu acknowledges the support through a grant of the Romanian Ministry of Research and Innovation, CNCS--UEFISCDI, project number PN-III-P4-ID-PCE-2016-0130,
within PNCDI III.


\begin{thebibliography}{99}

\bibitem{1} K. Andrews, K. Kuttler, M. Schillor,  Second order evolution equations with dynamic boundary conditions, {\it J. Math. Anal. Appl.} {\bf 197} (1996), 781-795.

\bibitem{2} V. Barbu, {\it Nonlinear Semigroups and Differential Equations in Banach Spaces}, Noordhoff, Leyden, The Netherlands, 1976.

\bibitem{4} V. Barbu, A. Favini, Existence for implicit differential equations in Banach spaces, {\it Atti Accad. Naz. Lincei Cl. Sci. Fiz. Mat. Natur. Rend. Mat. Appl.} {\bf 3} (1992), 203-215.

    \bibitem{3} V. Barbu, A. Favini,  Existence for an implicit differential equation, {\it Nonlinear Anal.} {\bf 32} (1998), 33-40.

\bibitem{5} E. DiBenedetto, R. Showalter, A pseudo-parabolic variational inequality and Stefan problem, {\it Nonlinear Anal.} {\bf 6} (1982), 279-291.

\bibitem{6} A. Favini, A. Yagi, Multivalued linear operators and degenerate evolution equations, {\it Annali Mat. Pura Appl.} {\bf 163} (1993), 353-384.

\bibitem{7} L. Gasinski, N.S. Papageorgiou, {\it Nonlinear Analysis}, Chapman \& Hall/CRC, Boca Raton, FL, 2006.

\bibitem{8} L. Gasinski, N.S. Papageorgiou, Anti-periodic solutions for nonlinear evolution inclusions, {\it J. Evolution Equations} {\bf 18} (2018), 1025-1047.

\bibitem{9} S. Hu, N.S. Papageorgiou, {\it Handbook of Multivalued Analysis. Volume I: Theory}, Kluwer Academic Publisher, Dordrecht, The Netherlands, 1997.

\bibitem{10} J.-L. Lions, {\it Quelques M\'ethodes de R\'esolution des Probl\`emes aux Limites Non Lin\'eaires}, Dunod, Paris, 1969.

\bibitem{11} Z. Liu, Existence for implicit differential equations with monotone perturbations, {\it Israel J. Math.} {\bf 129} (2002), 363-372.

\bibitem{12} N.S. Papageorgiou, F. Papalini, F. Renzacci, Existence of solutions and periodic solutions for nonlinear evolution inclusions, {\it Rend. Circ. Mat. Palermo} {\bf 48} (1999), 341-364.

\bibitem{preect} N.S. Papageorgiou, V.D. R\u adulescu, Periodic solutions for time-dependent subdifferential evolution inclusions, {\it Evol. Equations Control Theory} {\bf 6} (2017), 277-297.

\bibitem{13} N.S. Papageorgiou, V.D. R\u adulescu, D.D. Repov\v{s}, Sensitivity analysis for optimal control problems governed by nonlinear evolution inclusions, {\it Adv. Nonlinear Anal.} {\bf 6} (2017), 199-225.

\bibitem{14} R. Showalter, {\it Monotone Operators in Banach Spaces and Nonlinear Partial Differential Equations}, Math. Surveys and Monographs, Vol. 49, American Math. Soc., Providence, RI, 1997.

\bibitem{15} E. Zeidler, {\it Nonlinear Functional Analysis and its Applications II/B}, Springer, New York, 1990.

\end{thebibliography}
\end{document}